\documentclass[12pt]{amsart}

\usepackage{amssymb,amsmath,graphicx,color,textcomp, amsthm,bbm,bbold, enumerate,booktabs}
\usepackage{chngcntr}
\usepackage{apptools}
\usepackage{mathrsfs}
\usepackage{tikz}
\usetikzlibrary{trees}

\AtAppendix{\counterwithin{theorem}{section}}
\definecolor{Red}{cmyk}{0,1,1,0}

\definecolor{verde}{cmyk}{1,0,1,0}

\definecolor{loka}{cmyk}{.5,0,1,.5}
\definecolor{azul}{cmyk}{1,1,0,0}

\numberwithin{equation}{section}

\newcommand{\be}{\begin{equation}}
\newcommand{\ee}{\end{equation}}

\newtheorem{theorem}{Theorem}

\newtheorem{definition}{Definition}
\newtheorem{lemma}{Lemma}

\begin{document}
\title{On the fractional functional differential equation with abstract Volterra operator}
\author{J. Vanterler da C. Sousa$^1$}
\address{$^1$ Department of Applied Mathematics, Institute of Mathematics,
 Statistics and Scientific Computation, University of Campinas --
UNICAMP, rua S\'ergio Buarque de Holanda 651,
13083--859, Campinas SP, Brazil\newline
$^2$ Department of Mathematics, Shivaji University,\\ Kolhapur 416 004, Maharashtra, India\newline
e-mail: {\itshape \texttt{ra160908@ime.unicamp.br, capelas@ime.unicamp.br,kdkucche@gmail.com }}}

\author{E. Capelas de Oliveira$^1$}
\author{Kishor D. Kucche$^2$}

\begin{abstract} The present paper plans to examine the  existence, uniqueness and data dependence of the solution of the  fractional functional differential equation with the abstract operator of Volterra, in the context of the Picard operators. We present an example, in order to illustrate the results obtained. We present an application, with the end goal to illustrate the results obtained.

\end{abstract}
\maketitle

\section{Introduction}

In a recent special issue the editorial says \cite{Baleanu}: “The fractional and functional differential equations are two hot topics in mathematics, and they have multiple applications in various branches of science and engineering”. Likewise later is another extraordinary issue on qualitative theory of functional differential equation with the justification \cite{Tunc}: there are several areas where the equations play important rule. 

In 1990, Corduneanu \cite{oto1} investigated second-order functional differential equations involving operators of Volterra abstract. In this sense, around the year 2000 Corduneanu \cite{oto2} presented complete work on the existence and uniqueness of solutions of equations,  a general study on functional differential equations with abstract or causal or non-anticipatory Volterra operators.

On the other hand, in 2102 Otrocol \cite{oto} investigated the Ulam-Hyers and generalized Ulam-Hyers-Rassias stabilities of a differential equation
\begin{equation*}
x'(t)=f(t,x(t),V(x)(t)),\,\,%
\,t\in I:=[a,b]\subset\mathbb{R}
\end{equation*}
whose solution satisfies the following condition 
\begin{equation*}
x(a)=\alpha,\,\,\,x_{0}\in \mathbb{R}
\end{equation*}
where $\mathbb{R}$ is Banach space and $V:C([a,b],\mathbb{R})\rightarrow C([a,b],\mathbb{R})$ an abstract Volterra operator.

In a more recent paper Vanterler-Oliveira \cite{VanterlerOliveira}, using $\Psi-$Hilfer ($\Psi-$H) fractional derivative and the Banach fixed-point theorem, investigate the stabilities of Ulam-Hyers, Ulam-Hyers-Rassias and semi-Ulam-Hyers-Rassias for a particular class of fractional integro-differential equations by fixed point theorems \cite{jose1,jose2,jose3,jose4}. Additionally later is the paper by  Ali et al. \cite{Ali} where the existence of extremal solution to nonlinear boundary value problem of fractional order differential equations was discussed and the Ulam stability was investigated.

On the other hand, by means of the $\Psi-$H fractional derivative, a new operator in fractional calculus was introduced and, as an application the uniqueness of solutions for the nonlinear fractional Volterra integral equation (VIE) was presented \cite{VanterlerOliveira1}. Also, it is important to note that, in general, an analytical solution is very difficult, specifically if the equation is a fractional one. Thus, some indirect methods has been proposed. We mention one of them \cite{FanMophou}, where the authors propose a study of nonlinear problems for fractional differential equations via resolvent operators.

After this short summary and motivated by this theme, in this paper,  we will prove that a fractional functional differential equation with abstract Volterra operator Eq.(\ref{I1}) satisfying condition in Eq.(\ref{I1a}) is equivalent to a fractional Volterra integral equation. 
The existence, uniqueness and data dependence results for the solutions to fractional order functional differential equations with abstract Volterra operator (AVO) utilizing the $\Psi-$H fractional derivative \cite{VanterlerOliveira2} will be presented. As an application we will present a particular example to illustrate our result.

We consider the fractional functional differential equation with abstract Volterra operator of the form: 
\begin{equation}
{}^{\mathbf{H}}\mathfrak{D}_{a^{+}}^{\mu ,\nu ,\Psi }x(t)= 
\mathcal{U}(x)(t)+f(t,x(t)),\,\,%
\,t\in \Delta:=[a,b]  \label{I1}
\end{equation}%
whose solution satisfies the following condition 
\begin{equation}
\mathbf{I}_{a^{+}}^{1-\varepsilon ,\Psi }x(a)=x_{0},\,\,\,x_{0}\in \mathbb{R}  \label{I1a}
\end{equation}%
where $0<\mu \leq 1$,  $0 \leq \beta \leq 1$, $\varepsilon = \mu + \nu (1-\mu)$,  ${}^{\mathbf{H}}\mathfrak{D}_{a^{+}}^{\mu ,\nu ,\Psi }(\cdot )$ is the $\Psi-$H fractional derivative, $\mathbf{I}_{a^{+}}^{1-\varepsilon ,\Psi }(\cdot )$ is the $\Psi $-Riemann-Liouville fractional integral,  $\mathcal{U}:C_{\mu, \Psi, \xi }(\Delta\times \mathbb{R})\rightarrow
C_{\mu, \Psi, \xi }(\Delta\times \mathbb{R})$ is an abstract Volterra operator (AVO) and 
$f:C_{\mu, \Psi, \xi }(\Delta\times \mathbb{R})\rightarrow \mathbb{R}$ be an appropriate function specified later.

The primary inspiration for the elaboration of this paper is to introduce a work that permits a general perusing in the part of the investigation of  existence, uniqueness and data dependence of the fractional functional differential equation Eq.(\ref {I1})-Eq.(\ref {I1a}) with the abstract operator of Volterra, proposed by means of the $\Psi-$H fractional derivative.

The paper is organized as follows: In section 2, we present the definition  of $\psi-$H fractional derivative and the fundamental results about it.  We introduce the $(\mu,\Psi,\xi )$--Bielecki type norm and talk about its  special cases, Weakly Picard Operator. Likewise, we recall the definition of Picard's abstract operator and a few outcomes about it close this section. In section 3, we investigate the mains results: existence, uniqueness and data dependence of the solutions of the fractional functional differential equation with abstract Volterra operator. We present an example with the intention of illuminating the results obtained. Conclusions and the remarks closes the paper.

\section{Preliminaries}
\subsection{ $\Psi-$H fractional derivative:}
Let $\mu >0$, $\Delta\subset \mathbb{R}$ and $\Psi (t)$ be an increasing and positive monotone function on $(a,b]$, having a continuous derivative,
denoted by $\Psi ^{\prime }(t)$ on $\Delta$. The Riemann-Liouville fractional integral of a function $f$ with respect to another function $\Psi 
$ on $\Delta$ is defined by 
\begin{equation}
\mathbf{I}_{a+}^{\mu ,\Psi }x(t)=\frac{1}{\Gamma (\mu )}\int_{a}^{b}\mathcal{G}_{\Psi}^{\mu}(t,s)x(s)\,{\mbox{d}}s  \notag  \label{P7}
\end{equation}%
where  $\mathcal{G}_{\Psi }^{\mu }(t,s)=\Psi ^{\prime}(s)(\Psi(t)-\Psi(s))^{\mu-1}$ and $\Gamma(\cdot)$ is the gamma function, given by
\begin{equation}
\Gamma \left( z\right) =\int_{0}^{\infty }e^{-t}t^{z-1}dt, ~z\in \mathbb{C} ~\mbox{with} ~\mbox{Re}\left( z\right) >0.
\end{equation} The Riemann-Liouville fractional integral of a function $f$ with respect to another function $\mathbf{I}_{b-}^{\mu ,\Psi }\left( \cdot \right) $ is defined  analogously {\rm \cite{VanterlerOliveira2}}.

On the other hand, let $n-1<\mu \leq n$ with $n\in \mathbb{N}$, $\Delta=[a,b]$ be an interval such that $-\infty \leq a<b\leq +\infty $ and let $f,\Psi \in
C^{n}(\Delta,\mathbb{R})$ be two functions such that $\Psi $ is increasing and $\Psi ^{\prime }(t)\neq 0$, for $t\in \Delta$. Then the $\Psi-$H fractional derivative, denoted by ${}^{\mathbf{H}}\mathfrak{D}_{a^{+}}^{\mu ,\nu ,\Psi }(\cdot )$, of a function $f$ of order $\mu $ and type $\nu~ (0\leq \nu \leq 1)$, is defined by {\rm \cite{VanterlerOliveira2}} 
\begin{equation}
{}^{\mathbf{H}}\mathfrak{D}_{a+}^{\mu ,\nu ,\Psi }x(t)=\mathbf{I}_{a^{+}}^{\nu (n-\mu ),\Psi }\left( \frac{1}{\Psi ^{\prime }(t)}\frac{{\mbox{d}}}{{\mbox{d}}t}%
\right) ^{n}\mathbf{I}_{a^{+}}^{(1-\nu )(n-\mu ),\Psi }x(t).  \label{P8}
\end{equation}

The $\Psi-$H fractional derivative, denoted by ${}^{\mathbf{H}}\mathfrak{D}_{b-}^{\mu ,\nu ,\Psi }\left( \cdot \right) $ is defined analogously {\rm \cite{VanterlerOliveira2}}.

\begin{theorem}[\cite{VanterlerOliveira2}] \label{Thm1}
\label{T2} If $x\in C_{1-\varepsilon ,\Psi }^{1}(\Delta,\mathbb{R}),\,\,0<\mu \leq 1
$ and $0\leq \nu \leq 1$, then 
\begin{equation*}
\mathbf{I}_{a^{+}}^{\mu ,\psi }{}^{H}{\mathfrak{D}}_{a^{+}}^{\mu ,\nu ,\Psi
}x(t)=x(t)-\frac{(\Psi (t)-\Psi (a))^{\varepsilon -1}}{\Gamma (\varepsilon )}%
\mathbf{I}_{a^{+}}^{(1-\nu )(1-\mu ),\Psi }x(a).
\end{equation*}
\end{theorem}

\begin{theorem} [\cite{VanterlerOliveira2}] \label{Thm2}
\label{T3} Let $x\in C_{1-\gamma ,\Psi }^{1}(\Delta,\mathbb{R}),~\mu >0$ and $%
0\leq \nu \leq 1$, then we have 
\begin{equation*}
{\mathfrak{D}}_{a^{+}}^{\mu ,\nu ,\Psi }\mathbf{I}_{a^{+}}^{(1-\nu )(1-\mu
),\Psi }x(t)=x(t).
\end{equation*}
\end{theorem}

\subsection{ Weakly Picard Operator}	
 
The outcomes examined in this paper, are around the weakly Picard operator and some results derived from it. In this sense, we present a few definitions, lemmas and theorems in which they are fundamental for the development of the article and for that, we utilize the accompanying works  {\rm \cite{prin2,pri1}}.

Let $(\Omega ,d)$ be a metric space and $\mathbf{T}:\Omega \rightarrow \Omega $ be an operator. We denote by:
\begin{itemize}
\item [(a)] $\mathbf{F}_{\mathbf{T}}:=\{x\in \Omega  : \mathbf{T}\left( x\right) =x\}$ the fixed points set of $\mathbf{T}$;
\item [(b)] $\mathbf{I}_{\mathbf{T}}:=\left\{ Y\subset \Omega :\mathbf{T}\left( Y\right) \subset Y,\,\,\,Y\neq \varnothing \right\} $ the family of the nonempty invariant subsets of $\mathbf{T}$;
\item [(c)]  $\mathbf{T}^{n+1}=\mathbf{T}\circ \mathbf{T}^{n},\,\,\,\mathbf{T}^{0}=1_{x},\,\,\,\mathbf{T}^{1}=\mathbf{T},n\in \mathbb{N}$ the iterate operators of the operator $\mathbf{T}$.
\end{itemize}
\begin{definition}{\rm \cite{prin2,pri1}}  An operator $\mathbf{T}:\Omega \rightarrow \Omega $ is a Picard operator, if there exists $x^{\ast }\in \Omega $ satisfying the following conditions:
\begin{flushleft}
\begin{tabular}{cl}
{\rm (a)} & $\mathbf{F}_{\mathbf{T}}=\{x^{\ast }\}$; \\ 
{\rm (b)} & the sequence $\left( \mathbf{T}^{n}(x_{0})\right) _{n\in \mathbb{N}}$ converges to $x^{\ast }$ for all $x_{0}\in \Omega $.
\end{tabular}
\end{flushleft}
\end{definition}

\begin{definition} {\rm\cite{prin2,pri1}} An operator $\mathbf{T}:\Omega \rightarrow \Omega$ is a weakly Picard operator if the sequence $(\mathbf{T}^{n}(x))_{n\in \mathbb{N}}$ converges for all $x\in \Omega $ and its limit is a fixed point of $\mathbf{T}$.
\end{definition}

\begin{definition}{\rm\cite{prin2,pri1}} If $\mathbf{T}$ is a weakly Picard operator, then we consider the operator $\mathbf{T}^{\infty }$ defined by 
\begin{equation*}
\mathbf{T}^{\infty }:\Omega \rightarrow \Omega ,\quad \mathbf{T}^{\infty}(x)=\lim_{n\rightarrow \infty }\mathbf{T}^{n}(x).
\end{equation*}
Note that $A^\infty (\Omega)=\mathbf{F}_{\mathbf{T}}$.
\end{definition}

\begin{lemma} {\rm\cite{prin2,pri1}}\label{k1}
Let $(\Omega, d,\leq)$ be an ordered metric space and $\mathbf{T}: \Omega \to \Omega $ an operator. We suppose that:
\begin{enumerate}
\item[(i)] $\mathbf{T}$ is Weakly Picard Operator;
\item[(ii)] $\mathbf{T}$ is increasing.
\end{enumerate}
Then, the operator $\mathbf{T}^{\infty}$ is increasing.
\end{lemma}
\begin{lemma} {\rm\cite{prin2,pri1}} \label{k2}
(Abstract Gronwall lemma). Let $(\Omega, d,\leq)$ be an ordered metric space and $\mathbf{T}: \Omega \to \Omega $ an operator. We suppose that:
\begin{enumerate}
\item[(i)] $\mathbf{T}$ is Weakly Picard Operator;
\item[(ii)] $\mathbf{T}$ is increasing.
\end{enumerate}
If we denote by $x^{*}_{\mathbf{T}}$ the unique fixed point of $\mathbf{T}$, then:
\begin{enumerate}
\item[(a)] $x\leq \mathbf{T}(x)\, \Rightarrow \, x\leq {x_{ \mathbf{T}}^{*}}; $
\item[(b)]$x\geq\mathbf{T}(x)\, \Rightarrow \, x\geq {x_{ \mathbf{T}}^{*}} .$
\end{enumerate}
\end{lemma}
\begin{lemma} {\rm\cite{prin2,pri1}} \label{k4}
(Abstract comparison lemma). Let $(\Omega, \mathbf{d},\leq)$ be an ordered metric space and $\mathbf{T}_{1},\mathbf{T}_{2},\mathbf{T}_{3}: \Omega \to \Omega $ be such that:
\begin{enumerate}
\item[(i)] the operator $\mathbf{T}_{1},\mathbf{T}_{2},\mathbf{T}_{3}$ are Weakly Picard Operators ;
\item[(ii)]$\mathbf{T}_{1}\leq \mathbf{T}_{2}\leq \mathbf{T}_{3}$;
\item[(iii)] the operator $ \mathbf{T}_{2}$ is increasing.
\end{enumerate}
Then $ x \leq y \leq z$ implies that $\mathbf{T}^{\infty}_{1}(x)\leq \mathbf{T}^{\infty}_{2}(y) \leq \mathbf{T}^{\infty}_{3}(z)$.
\end{lemma}

\begin{definition} {\rm\cite{prin2,pri1}} Let $\mathbf{T}:\Omega \rightarrow \Omega $ be a weakly Picard operator and $c\in \mathbb{R}_{+}^{\ast }$. The operator $\mathbf{T}$ is $c$-weakly Picard operator if, and only if: 
\begin{equation}
\mathbf{d}(x,\mathbf{T}^{\infty }(x))\leq c\,\mathbf{d}(x,\mathbf{T}(x)),~ x\in \Omega .
\end{equation}
For the c-Picard Operators and c-Weakly Picard Operators we have the following lemma.
\end{definition}
\begin{lemma}  {\rm\cite{prin2,pri1}} \label{k5}
Let $(\Omega, \mathbf{d})$ be a metric space  and $\mathbf{T}_{1},\mathbf{T}_{2}: \Omega \to \Omega $ be two operators. We suppose that:
\begin{enumerate}
\item[(i)] $\mathbf{T}_{1}$ is $c-$ PO with $\textbf{F}_{\mathbf{T}_{1}}= {x^{*}_{\mathbf{T}_{1}}}$;
\item[(ii)]there exists $\eta \in \mathbb{R}^{*}_{+}$ such that $\mathbf{d}(\mathbf{T}_{1}(x), \mathbf{T}_{2}(x)) \leq \eta ,\, \forall x \in \Omega. $
\end{enumerate}
{\rm If} $x^{*}_{\mathbf{T}_{2}}\in \textbf{F}_{\mathbf{T}_{2}}$, then $d ( x^{*}_{\mathbf{T}_{2}},x^{*}_{\mathbf{T}_{1}}) \leq c\,\eta.$
\end{lemma}
\begin{lemma} \label{k6}
Let $(\Omega, d)$ be a metric space  and $\mathbf{T}_{1},\mathbf{T}_{2}: \Omega \to \Omega $ be two operators. We suppose that:
\begin{enumerate}
\item[(i)] the operators $\mathbf{T}_{1}$ and  $\mathbf{T}_{2}$ are $c-$Weakly Picard Operators; 
\item[(ii)]there exists  $\eta \in \mathbb{R}^{*}_{+}$ such that $ \mathbf{d}(\mathbf{T}_{1}(x), \mathbf{T}_{2}(x)) \leq \eta ,\, \forall x \in \Omega. $
\end{enumerate}
\end{lemma}

\begin{theorem}{\rm \cite{pri1,prin2}} \label{theorem1}
An operator $\mathbf{T}$ is a weakly Picard operator if, and only if, there exists a partition of  $\Omega $,~ $\Omega =\hspace{-0.3cm}
\begin{array}{c}
\bigcup  \\ 
\vspace{-0.4cm}{\lambda \in \Lambda }%
\end{array}%
\hspace{-0.2cm}\Omega _{{}_{\lambda }}$, such that:
\begin{flushleft}
\begin{tabular}{cl}
{\rm (a)} & $\Omega _{{}_{\lambda }}\in I(\mathbf{T}),~ \forall\, \lambda \in
\Lambda $; \\ 
{\rm (b)} & $\mathbf{T}|_{\Omega _{{}_{\lambda }}}:\Omega _{{}_{\lambda }}\rightarrow \Omega _{{}_{\lambda }}$ is a Picard operator,  $~ \forall\, \lambda
\in \Lambda $.
\end{tabular}
\end{flushleft}
\end{theorem}

For more details on the concept of Picard operator, Picard weak, Picard $c$-weak and its related fundamental work, we suggest \cite{pri1,prin2,Rus1,Rus2}.

\subsection{ $(\mu,\Psi,\xi )$--Bielecki type norm:}

The Mittag-Leffler functions of one, two and of different parameters are to be sure vital for the investigation of several problems in the field of mathematics, physics, biology among other areas, in particular these functions provide solutions of fractional differential equations, and even more generalizes the exponential function.  We recall the definitions of Mittag-Leffler function of one and two parameters.

\begin{definition}{\rm\cite{sao}}\label{def1}{\rm (One parameter Mittag-Leffler function)}. The Mittag-Leffler function is given by the series
\begin{equation}\label{A1}
\mathcal{E}_{\mu }\left( z\right) =\overset{\infty }{\underset{k=0}{\sum }}\frac{z^{k}}{\Gamma \left( \mu k+1\right) },
\end{equation}
where $\mu \in \mathbb{C}$ with $ {Re}\left( \mu \right) >0$.
\end{definition}

\begin{definition}{\rm\cite{sao}}\label{def2}{\rm (Two parameters Mittag-Leffler function)}. The two parameters Mittag-Leffler function is given by the series
\begin{equation}\label{A2}
\mathcal{E}_{\mu ,\nu }\left( z\right) =\overset{\infty }{\underset{k=0}{\sum }}\frac{z^{k}}{\Gamma \left( \mu k+\nu \right) },
\end{equation}
where $\mu ,\nu \in \mathbb{C}$, ${Re}\left( \mu \right) >0$ and ${Re}\left( \nu\right) >0$.
\end{definition}
In particular, for $\nu=1$ we have $\mathcal{E}_{\mu ,1}\left( z\right)=\mathcal{E}_{\mu}\left( z\right) $. On the other
hand, taking $\nu=\mu=1$, we have $\mathcal{E}_{1,1}\left( z\right)=\mathcal{E}_{1}\left( z\right)=e^{z}$.

For more details on Mittag-Leffler function and its other formulations with more parameters, we suggest \cite{mitt1}.

With this foundation, we now present the concept of $(\mu,\Psi,\xi )$--Bielecki type norm which is given in terms of the Mittag-Leffler function,
 an increasing and positive monotone function $\Psi: \Delta \to \mathbb{R_+}$  and constant $\xi >0$. It will be important tool for the investigation of the main results of this article.

Let $\Delta=[a,b]$ be a finite interval of the half-axis $\mathbb{R}^{+}$ and let $C(\Delta,\mathbb{R})$ be the space of continuous functions $f$ on $\Delta$ with the norm {\rm\cite{S1}}
\begin{equation}
\left\Vert x\right\Vert =\sup_{t\in\Delta}\left | x\left( t\right)\right | ,\quad x\in C\left(\Delta ,\mathbb{R}\right) .
\end{equation}

Let $\xi, ~\mu >0 ~(\xi, ~ \mu \in \mathbb{R})$ and   $\Psi: \Delta \to \mathbb{R_+}$ be  a positive  and an increasing monotone function. Consider the space $C_{\xi, \,\mu, \, \Psi }\left( \Delta,\mathbb{R}\right) $  of all continuous functions $x\in C\left( \Delta,\mathbb{R}\right) $, such that {\rm\cite{S1}}
\begin{equation*}
\sup_{t\in \Delta}\frac{\left| x(t)\right| }{\mathcal{E}_{\mu }\left( \xi \left( \Psi \left( t\right) -\Psi \left( a\right) \right) ^{\mu }\right) }
<\infty 
\end{equation*}
where $\mathcal{E}_{\mu }:\mathbb{R}\rightarrow \mathbb{R}$ is the one-parameter Mittag-Leffler function.

We couple the linear space $C_{\xi, \,\mu, \, \Psi }\left( \Delta,\mathbb{R}\right)$ 
with the norm defined by 
\begin{equation}
\left\Vert x\right\Vert _{\xi, \,\mu, \, \Psi }:=\sup_{t\in \Delta}\frac{\left| x\left(t\right) \right| }{\mathcal{E}_{\mu }(\xi (\Psi (t)-\Psi (a))^{\mu})}.  \label{P2}
\end{equation}
This norm induces a metric 
\begin{equation}
\mathbf{d}_{\xi, \,\mu, \, \Psi }(x,y)=\left\Vert x-y\right\Vert _{\xi, \,\mu, \, \Psi }=\sup_{t\in \Delta}\frac{\left | x(t)-y(t)\right | }{\mathcal{E}_{\mu }\left( \xi \left( \Psi \left( t\right) -\Psi \left(
a\right) \right) ^{\mu }\right) }.  \label{P1} 
\end{equation}

Note that the  norm and metric  as seen in {\rm Eq.(\ref{P2})} and {\rm Eq.(\ref{P1})} are in fact an extension of a class of norms and metrics, that is, taking $\Psi
(t)=t$ in {\rm Eq.(\ref{P2})} and {\rm Eq.(\ref{P1})}, 
the norm defined by 
\begin{equation}
\left\Vert x\right\Vert _{\xi, \,\mu,}=\sup_{t\in \Delta}\frac{\left | x\left(t\right) \right | }{\mathcal{E}_{\mu } \left(  \xi \left( t-a\right)
^{\mu }\right)  }.  \label{P4}
\end{equation}
and the metric
\begin{equation}
\mathbf{d}_{\xi, \,\mu}(x,y)=\sup_{t\in \Delta}\frac{|x(t)-y(t)|}{\mathcal{E}_{\mu }\left( \xi \left( t-a\right) ^{\mu }\right) }.   \label{P3}
\end{equation}

On the other hand, taking $\Psi (t)=t$ and applying the limit $\mu \rightarrow 1$ in {\rm Eq.(\ref{P2})} and {\rm Eq.(\ref{P1})}, 
we get the Bielecki norm {\rm\cite{prin2}} defined by 
\begin{equation}
\left\Vert x\right\Vert _{\xi, \,1}=\sup_{t\in \Delta}\frac{\left | x\left(t\right) \right | }{\mathrm{exp}\left( \xi \left( t-a\right) \right)  }
\label{P6}
\end{equation}
with a induced metric
\begin{equation}
\mathbf{d}_{\xi, \,1}(x,y)=\sup_{t\in \Delta}\frac{\left| x(t)-y(t)\right| }{\mathrm{exp}\left(  \xi \left( t-a\right) \right) }. \label{P5}
\end{equation}%

\begin{lemma}{\rm \cite{S1}}  If Let $\xi, ~\mu >0 ~(\xi, ~ \mu \in \mathbb{R})$ and   $\Psi: \Delta \to \mathbb{R_+}$ be  a positive  and an increasing monotone function, then: 

\begin{enumerate}
\item $\left\Vert  \cdot  \right\Vert _{\xi, \,\mu, \, \Psi }$ is a norm; 
\item $\left( C_{\xi, \,\mu, \, \Psi }\left( \Delta,\mathbb{R}\right) ,\left\Vert \left( \cdot\right) \right\Vert _{\xi, \,\mu, \, \Psi }\right) $ is a Banach space;

\item $\mathbf{d}_{\xi, \,\mu, \, \Psi }$ is a metric; 
	
\item $\left( C_{\xi, \,\mu, \, \Psi }\left( \Delta,\mathbb{R}\right),\, \mathbf{d}_{\xi, \,\mu, \, \Psi }\right) $ is a complete metric space. 

\end{enumerate}
\end{lemma}	


\section{Main results}

To obtain our main results we need the following conditions on $f$ and  $\mathcal{U}$ :

\begin{itemize}
\item [(C$_{1}$)]
There exists $L_{f}>0$ such that 
\begin{equation}
\left\vert f\left( t,u_{1}\right) -f\left( t,u_{2}\right) \right\vert \leq L_{f}\left\vert u_{1}-u_{2}\right\vert ,~ t\in\Delta,\quad u_{1},u_{2}\in \mathbb{R};  \label{C1}
\end{equation}
\item [(C$_{2}$)] $\mathcal{U}:C_{\xi, \,\mu, \, \Psi }\left( \Delta ,\mathbb{R}\right) \rightarrow C_{\xi, \,\mu, \, \Psi }\left( \Delta ,\mathbb{R}\right) $ is an AVO and there exists $L_{\mathcal{U}}>0$ and $\xi >0$ with $\xi>L_{\mathcal{U}}+L_{f}$, such that 
\begin{equation}
\left| \mathcal{U}(x)(t)-\mathcal{U}(y)(t)\right|\leq L_{\mathcal{U}}\left | x(t)-y(t)\right|,~ x,y\in C_{\xi, \,\mu, \, \Psi }\left( \Delta ,\mathbb{R}\right),~t \in \Delta \label{C2}
\end{equation}%
where $\left\Vert  \cdot  \right\Vert _{\xi, \,\mu, \, \Psi }$ is the Bielecki type norm defined by {\rm Eq.(\ref{P2})}.
\end{itemize}

In this section, we will research the principle motivation behind this paper, i.e., the existence, uniqueness and date dependence of the solution of the fractional functional differential equation with the abstract Volterra operator, introduced by means of the $\Psi-$H fractional derivative.

First, let's prove that the fractional functional differential equation, {\rm Eq.(\ref{I1})}, satisfying {\rm Eq.(\ref{I1a})}, is equivalent to the following fractional VIE
\begin{equation}\label{sol1}
x(t)=\mathcal{R}^{\Psi}_{\varepsilon}\left( t,a\right)  x_{0}+\frac{1}{\Gamma(\mu )}\int_{a}^{t}\mathbf{\mathcal{G}_{\Psi }^{\mu }}(t,s)\mathcal{U}(x)(s)\,{d}s+\frac{1}{\Gamma(\mu )}\int_{a}^{t}\mathbf{\mathcal{G}_{\Psi }^{\mu }}(t,s)f(s,x(s))\,{d}s
\end{equation}
where $\mathcal{R}^{\Psi}_{\varepsilon}\left( t,a\right) :=\dfrac{(\Psi (t)-\Psi (a))^{\varepsilon-1}}{\Gamma (\varepsilon )}$ with $t\in\Delta.$

Let $x\in C_{\xi, \,\mu, \, \Psi }^{1}\left( \Delta,\mathbb{R}\right) $  is the solution of {\rm Eq.(\ref{I1})}--{\rm Eq.(\ref{I1a})}. Then using the {\rm Theorem \ref{Thm1}, we have
\begin{equation}
\mathbf{I}_{a^{+}}^{\mu ,\Psi }{}^{\mathbf{H}}\mathfrak{D}_{a^{+}}^{\mu ,\nu ,\Psi}x(t)=\mathbf{I}_{a^{+}}^{\mu ,\Psi }\left( \mathcal{U}(x)(t)+f(t,x(t))\right) 
\end{equation}
which implies 
\begin{equation}
x(t)=\mathcal{R}^{\Psi}_{\varepsilon}\left( t,a\right) x_{0}+\mathbf{I}_{a^{+}}^{\mu ,\Psi }\mathcal{U}(x)(t)+\mathbf{I}_{a^{+}}^{\mu ,\Psi }f(t,x(t)).  \label{MR1}
\end{equation}
This proves $x\in C_{\xi, \,\mu, \, \Psi }^{1}(\Delta,\mathbb{R})$ is also a solution of {\rm Eq.(\ref{I1})}.

On the other hand, by applying the fractional derivative ${}^{\mathbf{H}}\mathfrak{D}_{a^{+}}^{\mu ,\nu ,\Psi }\left( \cdot \right) $ on both sides of {\rm Eq.(\ref{MR1})}, we get 
\begin{eqnarray}
{}^{\mathbf{H}}\mathfrak{D}_{a^{+}}^{\mu ,\nu ,\Psi }x(t) &=&\text{}^{\mathbf{H}}\mathfrak{D}_{a^{+}}^{\mu ,\nu ,\Psi }\left[ \mathcal{R}^{\Psi}_{\varepsilon}\left( t,a\right) x_{0}+\mathbf{I}_{a^{+}}^{\mu ,\Psi }f(t,x(t))\right]   \notag \\
&=&\text{}^{\mathbf{H}}\mathfrak{D}_{a^{+}}^{\mu ,\nu ,\Psi }\mathcal{R}^{\Psi}_{\varepsilon}\left( t,a\right) x_{0}+\mathcal{U}(x)(t)+f(t,x(t))  \notag \\
&=&\mathcal{U}(x)(t)+f(t,x(t)).
\end{eqnarray}

Also, if $x\in C_{\xi, \,\mu, \, \Psi }^{1}(\Delta,\mathbb{R})$ is a solution of {\rm Eq.(\ref{I1})}, then $x$ is a solution of 
\begin{equation}
x(t)=\mathcal{R}^{\Psi}_{\varepsilon}\left( t,a\right) \mathbf{I}_{a^{+}}^{1-\varepsilon ,\Psi}x(a)+\mathbf{I}_{a^{+}}^{\mu ,\Psi }(x)(t)+\mathbf{I}_{a^{+}}^{\mu ,\Psi}f(t,x(t)).  \label{MR2}
\end{equation}
and if $x\in C_{\xi, \,\mu, \, \Psi }\left( \Delta,\mathbb{R}\right)$ is a solution of {\rm Eq.(\ref{MR2})} then it is solution of {\rm Eq.(\ref{I1})}.

Consider the following operators $\Theta _{f},\Phi _{f}:C_{\xi, \,\mu, \, \Psi }\left( \Delta,\mathbb{R}\right)\rightarrow C_{\xi, \,\mu, \, \Psi }\left( \Delta,\mathbb{R}\right)$ defined by 
\begin{eqnarray} \label{MR3}
\Theta _{f}(x)(t) &:&=\mathcal{R}^{\Psi}_{\varepsilon}\left( t,a\right) x_{0}+\frac{1}{\Gamma (\mu )}\int_{a}^{t}\mathcal{G}_{\Psi }^{\mu }(t,s)\mathcal{U}(x)(s)ds +\frac{1}{\Gamma (\mu )}\mathcal{G}_{\Psi }^{\mu }(t,s)f(s,x(s))ds. \notag \\
\end{eqnarray}
and 
\begin{eqnarray}\label{MR4}
\Phi _{f}(x)(t) &:&=\mathcal{R}^{\Psi}_{\varepsilon}\left( t,a\right) \mathbf{I}_{a^{+}}^{1-\varepsilon,\Psi }x(a)+\frac{1}{\Gamma (\mu )}\int_{a}^{t}\mathcal{G}_{\Psi }^{\mu }(t,s)\mathcal{U}(x)(s)\,{d}s+\frac{1}{\Gamma (\mu )}\int_{a}^{t}\mathcal{G}_{\Psi }^{\mu }(t,s)f(s,x(s))\,{d}s. \notag \\
\end{eqnarray}

\begin{theorem}\label{theorem2} Let $f:C_{\xi, \,\mu, \, \Psi }\left( \Delta,\mathbb{R}\right) \to \mathbb{R}$, $x_0 \in \mathbb{R}$ and suppose that the conditions $(C_1)$ and $(C_2)$  given in {\rm Eq.(\ref{C1})} and {\rm Eq.(\ref{C2})} respectively,  are satisfied. Then,
\begin{flushleft}
\begin{tabular}{cl}
{\rm (a)} & the {\rm Eq.(\ref{I1})}--{\rm Eq.(\ref{I1a})} has a unique solution in $C_{\xi, \,\mu, \, \Psi }\left( \Delta,\mathbb{R}\right)$;\\
{\rm (b)} & $\Theta_f$ is a Picard operator in $C_{\xi, \,\mu, \, \Psi }\left( \Delta,\mathbb{R}\right)$;\\
{\rm (c)} & $\Phi_f$ is weakly Picard operator in $C_{\xi, \,\mu, \, \Psi }\left( \Delta,\mathbb{R}\right)$.
\end{tabular}
\end{flushleft}
\end{theorem}

\begin{proof} Let $\Omega=C_{\xi, \,\mu, \, \Psi }\left( \Delta,\mathbb{R}\right)$ and consider the set defined by 
\begin{equation*}
\Omega _{x_{0}}=\left\{ x\in \Omega: \mathbf{I}_{a^{+}}^{1-\varepsilon;\Psi }x(a)=x_{0}\right\}, x_{0}\in \mathbb{R}.
\end{equation*}

We remark that $\Omega =\hspace{-0.4cm}
\begin{array}{c}
\bigcup  \\ 
\vspace{-0.4cm}{x_{0}\in \mathbb{R}}%
\end{array}%
\hspace{-0.2cm}\Omega _{x_{0}}$, is a partition of $\Omega$ and

\begin{flushleft}
\begin{tabular}{cl}
{\rm 1.} & $\Theta _{f}(\Omega )\subset \Omega _{x_{0}}$ and $E_{f}(\Omega _{x_{0}})\subset \Omega _{x_{0}}$, for all $x_{0}\in \mathbb{R}$;\\ 
{\rm 2.} & $\Theta _{f}|_{\Omega _{x_{0}}}=\Phi _{f}|_{\Omega _{x_{0}}}$, for all $x_{0}\in \mathbb{R}$.%
\end{tabular}
\end{flushleft}

Using equation {\rm Eq.(\ref{MR3})}, for any $x,y\in \Omega $ and $t \in \Delta$, we have
\begin{eqnarray} \label{cn1}
&&\left| \Theta _{f}\left( x\right)(t) -\Theta _{f}\left( y\right)(t)\right| \notag\\
 &=&\left| \frac{1}{\Gamma (\mu )}\int_{a}^{t}\mathcal{G}_{\Psi }^{\mu }(t,s)\mathcal{U}(x)(s)\,\mathrm{d}s-\frac{1}{\Gamma (\mu )}\int_{a}^{t}\mathcal{G}_{\Psi }^{\mu }(t,s)\mathcal{U}(y)(s)\,\mathrm{d}s\right|   \notag\\
&&+\left| \frac{1}{\Gamma (\mu )}\int_{a}^{t}\mathcal{G}_{\Psi }^{\mu }(t,s)f(s,x(s))\,\mathrm{d}s-\frac{1}{\Gamma (\mu )}\int_{a}^{t}\mathcal{G}_{\Psi }^{\mu }(t,s)f(s,y(s))\,\mathrm{d}s\right|   \notag \\
&\leq &\frac{1}{\Gamma (\mu )}\int_{a}^{t}\mathcal{G}_{\Psi }^{\mu }(t,s)\left| \mathcal{U}(x)(s)-\mathcal{U}(y)(s)\right|\,\mathrm{d}s+\frac{1}{\Gamma (\mu )}\int_{a}^{t}\mathcal{G}_{\Psi }^{\mu }(t,s)\left | f(s,x(s))-f(s,y(s))\right| \mathrm{d}s  \notag \\
&\leq &\frac{L_{\mathcal{U}}}{\Gamma (\mu )}\int_{a}^{t}\mathcal{G}_{\Psi }^{\mu }(t,s)\left| x(s)-y(s)\right|\,\mathrm{d}s+\frac{L_{f}}{\Gamma (\mu )}\int_{a}^{t}\mathcal{G}_{\Psi }^{\mu }(t,s)\left| x(s)-y(s)\right|\,\mathrm{d}s  \notag \\
&\leq &\frac{(L_{\mathcal{U}}+L_{f})}{\Gamma(\mu )}\int_{a}^{t}\mathcal{G}^{\mu}_{\Psi}\left( t,s\right)\, {\mathcal{E}_{\mu }\left( \xi \left( \Psi \left( s\right) -\Psi \left( a\right) \right) ^{\mu }\right) }\,\frac{\left| x(s)-y(s)\right|}{{\mathcal{E}_{\mu }\left( \xi \left( \Psi \left( s\right) -\Psi \left( a\right) \right) ^{\mu }\right) }}\,\mathrm{d}s  \notag \\
&\leq&\frac{(L_{f}+L_{\mathcal{U}})\left\Vert x-y\right\Vert_{\xi, \,\mu, \, \Psi }}{\Gamma (\mu)}\,\int_{a}^{t}\mathcal{G}^{\mu}_{\Psi}\left( t,s\right)\, {\mathcal{E}_{\mu }\left( \xi \left( \Psi \left( s\right) -\Psi \left( a\right) \right) ^{\mu }\right) }\,\mathrm{d}s \notag\\
&=& (L_{f}+L_{\mathcal{U}})\left\Vert x-y\right\Vert_{\xi, \,\mu, \, \Psi }\, \mathbf{I}_{a^{+}}^{\mu ,\Psi }{\mathcal{E}_{\mu }\left( \xi \left( \Psi \left( t\right) -\Psi \left( a\right) \right) ^{\mu }\right)}.
\end{eqnarray}
Note that 
\begin{align*}
\mathbf{I}^{\mu;\,\Psi}_{a+}\mathcal{E}_{\mu}(\xi(\Psi(t)-\Psi(a))^\mu)
&=\mathbf{I}^{\mu;\,\Psi}_{a+}\left[\sum_{k=o}^{\infty}\frac{(\xi(\Psi(t)-\Psi(a)^{\mu}))^k}{\Gamma(k\mu+1)}\right]\\
&=\sum_{k=o}^{\infty}\frac{\xi^k}{\Gamma(k\mu+1)}\,\,\mathbf{I}^{\mu;\,\Psi}_{a+}(\Psi(t)-\Psi(a))^{k\mu}\\
&=\sum_{k=o}^{\infty}\frac{\xi^k}{\Gamma(k\mu+1)}\, \frac{\Gamma(k\mu+1)}{\Gamma(k\mu+\mu+1)}(\Psi(t)-\Psi(a))^{k\mu+\mu}\\
&=\sum_{k=o}^{\infty}\frac{\xi^k}{\Gamma(\mu(k+1)+1)}(\Psi(t)-\Psi(a))^{\mu(k+1)}\\
&=\frac{1}{\xi}\,\sum_{k=o}^{\infty}\frac{\xi^{k+1}}{\Gamma(\mu(k+1)+1)}(\Psi(t)-\Psi(a))^{\mu(k+1)}\\
&=\frac{1}{\xi}\,\left\lbrace \sum_{k=o}^{\infty}\frac{\xi^{k}}{\Gamma(k\mu+1)}(\Psi(t)-\Psi(a))^{k\mu}-1\right\rbrace\\
&=\frac{1}{\xi}\,\left\lbrace \mathcal{E}_{\mu}(\xi(\Psi(t)-\Psi(a))^\mu)-1 \right\rbrace.
\end{align*}

Thus from equation {\rm Eq.(\ref{cn1})}, for any $x,y\in \Omega $ and $t \in \Delta$, we have
\begin{eqnarray*}
\left| \Theta _{f}\left( x\right)(t) -\Theta _{f}\left( y\right)(t)\right| 
\leq \frac{(L_{f}+L_{\mathcal{U}})\left\Vert x-y\right\Vert_{\xi, \,\mu, \, \Psi }}{\xi} \,\left\lbrace \mathcal{E}_{\mu}(\xi(\Psi(t)-\Psi(a))^\mu)-1 \right\rbrace. \notag
\end{eqnarray*}

Therefore, for any $x,y\in \Omega $, 
\begin{eqnarray*}
\left\Vert \Theta _{f}\left( x\right) -\Theta _{f}\left( y\right)\right\Vert _{\xi, \,\mu, \, \Psi } &=&\sup_{t\in \Delta}\frac{\left| \Theta _{f}\left( x\right)(t) -\Theta _{f}\left( y\right)(t)\right| }{\mathcal{E}_{\mu }(\xi (\Psi (t)-\Psi (a))^{\mu})}\\
&\leq& \frac{(L_{f}+L_{\mathcal{U}})}{\xi} \,\left\lbrace 1-\frac{1}{\mathcal{E}_{\mu }(\xi (\Psi (b)-\Psi (a))^{\mu})} \right\rbrace \, \left\Vert x-y\right\Vert_{\xi, \,\mu, \, \Psi }.\notag
\end{eqnarray*}

Since $\Psi$ is increasing function, $\xi>0$, 
$
0\leq \left\lbrace 1-\frac{1}{\mathcal{E}_{\mu }(\xi (\Psi (b)-\Psi (a))^{\mu})} \right\rbrace\leq 1
$
and hence,  we have  
\begin{eqnarray*}
\left\Vert \Theta _{f}\left( x\right) -\Theta _{f}\left( y\right)\right\Vert _{\xi, \,\mu, \, \Psi } &\leq& \frac{(L_{f}+L_{\mathcal{U}})}{\xi} \, \left\Vert x-y\right\Vert_{\xi, \,\mu, \, \Psi }, ~x,y\in \Omega. \notag
\end{eqnarray*}

Choosing $\xi >0$ such that $L_{\mathcal{U}}+L_{f}<\xi $, we have that $\Theta _{f}$ is a contraction in $(\Omega,\left\Vert \cdot  \right\Vert _{\xi, \,\mu, \, \Psi })$. Hence part (a) and (b) of the Theorem \ref{theorem2} is proved. Moreover, the operator $\Phi _{f}\left\vert _{\Omega _{x_{0}}}\right. :\Omega _{x_{0}}\rightarrow \Omega _{x_{0}}$ is a contraction and by means of theorem of weakly Picard operator {\rm (Theorem\ref{theorem1})}, we get that $\Phi _{f}$ is $c$-weakly Picard operator with  
\begin{equation}
c=\left( 1-\frac{(L_{f}+L_{\mathcal{U}})}{\xi}\right) ^{-1}  \label{jose}
\end{equation}%
which complete the proof.
\end{proof}

The natural question that arises is the following: ``Will it be possible to relate a solution of {\rm Eq.(\ref{I1})} and {\rm Eq.(\ref{I1a})} and a solution of same problem?" The appropriate response is yes and will be given by the theorem exhibited and demonstrated below.

\begin{theorem} [Theorem of $\check{C}$aplygin type] \label{t3}
Consider the function $f : C_{\xi, \,\mu, \, \Psi }\left( \Delta,\mathbb{R}\right) \to \mathbb{R}$, $x_0 \in \mathbb{R}$ and suppose that
\begin{flushleft}
\begin{tabular}{cl}
{\rm (a)} & the conditions $(C_1)$ and $(C_2)$ are satisfied; \\ 
{\rm (b)} & $f(t,\cdot) : \mathbb{R} \to \mathbb{R}$ is increasing; \\ 
{\rm (c)} & $\mathcal{U} : C_{\xi, \,\mu, \, \Psi }\left( \Delta,\mathbb{R}\right) \to C_{\xi, \,\mu, \, \Psi }\left( \Delta,\mathbb{R}\right)$ is increasing.
\end{tabular}
\end{flushleft}
Besides that, let $x$ be a solution of {\rm Eq.(\ref{I1})} and $y$ a satisfying the inequality 
\begin{equation}
{}^{\mathbf{H}}\mathfrak{D}_{a^{+}}^{\mu ,\nu ,\Psi }y(t)\leq \mathcal{U}(y)(t)+f(t,y(t)),\quad t\in\Delta.  \label{MR5}
\end{equation}

Then $\mathbf{I}_{a^{+}}^{1-\varepsilon ,\Psi }y(a)\leq \mathbf{I}_{a^{+}}^{1-\varepsilon ,\Psi }x(a)$ implies $y\leq x$.
\end{theorem}
\begin{proof}
Consider the  relations $x=\Theta_f(x)$ and $y\leq \Phi_f(y)$ and $f : C_{\xi, \,\mu, \, \Psi }(\Delta,\,\mathbb{R}) \to \mathbb{R}$. By  means of the conditions  $x_0 \in \mathbb{R}$, $f : C_{\xi, \,\mu, \, \Psi }\left( \Delta,\mathbb{R}\right) \to \mathbb{R}$ and the conditions $(C_1)$ and $(C_2)$ it follows: the operator $\Theta_f$ is weakly Picard operator also, from conditions $(b)$ and $(c)$ we have $\Phi_f$ is an increasing operator. Since $x_{0}\in \mathbb{R}, ~f:C_{\mu, \Psi, \xi }(\Delta\times \mathbb{R})\rightarrow \mathbb{R}$ By means of the {\rm Lemma \ref{k1} } , we have that $\Phi_f^{\infty}(\cdot)$ is increasing. Now, consider $x_0 \in \mathbb{R}$ and denote by $ \widetilde{x}_0$ the  function defined by 
\begin{equation}  \label{MR6}
\widetilde{x}_0 : \Delta \to \mathbb{R}, \quad \widetilde{x}_0(t) = x_0, \quad t \in \Delta.
\end{equation}
Using the {\rm Theorem \ref{theorem2}}, we obtain $\Phi_f(X_{x_0}) \subset X_{x_0}$, $x_0 \in \mathbb{R}$, $\left. \Phi_f\right|_{X_{x_0}}$ is a contraction
and since $\widetilde{x}_0 \in {X_{x_0}}$ then $\Phi_f^{\infty}(\widetilde{x}_0) = \Phi_f^{\infty}(y)$, $y \in X_{x_0}$, where $\Phi_f^{\infty} (\widetilde{x}_0) = \displaystyle \lim_{n \to \infty} \Phi_f^n (\widetilde{x}_0)$.

For finalize the proof let $y\leq \Phi _{f}(y)$. Since $\Phi _{f}$ is increasing by using the {\rm abstract Gronwall Lemma \ref{k4}} we get $y\leq \Phi _{f}^{\infty }(y)$. Also, $y,~\mathbf{I}_{a^{+}}^{1-\varepsilon ,\Psi }\widetilde{y}(a)\in \Omega {\mathbf{I}_{a^{+}}^{1-\varepsilon ,\Psi }\widetilde{y}(a)}$, so $\Phi _{f}^{\infty }(y)=\Phi _{f}^{\infty }\left( \mathbf{I}_{a^{+}}^{1-\varepsilon,\Psi }\widetilde{(}y)(a)\right) $. But, $\mathbf{I}_{a^{+}}^{1-\varepsilon ,\Psi }y(a)\leq \mathbf{I}_{a^{+}}^{1-\varepsilon ,\Psi }x(a)$, $\Phi _{f}^{\infty }$ is increasing and $\Phi _{f}^{\infty }\left( \mathbf{I}_{a^{+}}^{1-\varepsilon ,\Psi }x(a)\right) =\Phi _{f}^{\infty }(x)=x$. Therefore, 
\begin{equation*}
y\leq \Phi _{f}^{\infty }(y)=\Phi _{f}^{\infty }\left( \mathbf{I}_{a^{+}}^{1-\varepsilon,\Psi }\widetilde{y}(a)\right) \leq \Phi _{f}^{\infty }\left(\mathbf{I}_{a^{+}}^{1-\varepsilon ,\Psi }\widetilde{x}(a)\right) =x
\end{equation*}
which close the proof. 
\end{proof}

\begin{theorem} [Comparison theorem] \label{t4} Let $f_i \in C_{\xi, \,\mu, \, \Psi }\left( \Delta,\mathbb{R}\right) $, $i=1,2,3$ satisfy the conditions ${\rm (C_1)}$ and ${\rm (C_2)}$. In addition, suppose that
\begin{flushleft}
\begin{tabular}{ll}
{\rm (a)}  $f_1 \leq f_2 \leq f_3$ and $\mathcal{U}_1 \leq \mathcal{U}_2 \leq \mathcal{U}_3$;\\
{\rm (b)}  $f_2(t,\cdot) : \mathbb{R} \to \mathbb{R}$ is increasing;\\
{\rm (c)}  $\mathcal{U}_2 : C_{\xi, \,\mu, \, \Psi }\left( \Delta,\mathbb{R}\right) \to C_{\xi, \,\mu, \, \Psi }\left( \Delta,\mathbb{R}\right)$ is increasing;\\
{\rm (d)}  $x_{i}\in C_{\xi, \,\mu, \, \Psi }^{1}(\Delta,\mathbb{R})$ be a solution of the equation \\
\end{tabular}
\end{flushleft}
\begin{equation*}
{}^{\mathbf{H}}\mathfrak{D}_{a^{+}}^{\mu ,\nu;\Psi} x_{i}(t)=\mathcal{U}_{i}(x)(t)+f_{i}(t,x(t)),\,\,\,t\in I:=\Delta,~\rm{and}~i=1,2,3.
\end{equation*}
\rm{If} ~$\mathbf{I}_{a^{+}}^{1-\varepsilon ,\Psi }x_{1}(a)\leq \mathbf{I}_{a^{+}}^{1-\varepsilon ,\Psi }x_{2}(a)\leq \mathbf{I}_{a^{+}}^{1-\varepsilon ,\Psi }x_{3}(a)$, then $x_{1}\leq x_{2}\leq x_{3}$.
\end{theorem}

\begin{proof} Note that $\Phi_{f_i}$ with $i=1,2,3$, are weakly Picard operators by means of {Theorem \ref{theorem2}}. In addition, using the condition {\rm (b)} the operator $\Phi_{f_2}$ is monotone increasing. In this sense, using the condition {\rm (a)}, we get $\Phi_{f_1} \leq \Phi_{f_2} \leq \Phi_{f_3}$.

Now, we define $\widetilde{x}_i(a)(t) = x_i (a)$ for $t \in \Delta$, with $\widetilde{x}_i(a) \in   C_{\xi, \,\mu, \, \Psi }\left( \Delta,\mathbb{R}\right) $. In this sense, we get
$$
I_{a^{}}^{1-\varepsilon,\Psi} \widetilde{x}_1(a)(t) \leq I_{a^{}}^{1-\varepsilon,\Psi} \widetilde{x}_2(a)(t) \leq I_{a^{}}^{1-\varepsilon,\Psi} \widetilde{x}_3(a)(t), ~ t \in \Delta. 
$$

By means of the {\rm Lemma \ref{k4}}, we obtain
$$
\Phi_{f_1}^{\infty} \left( I_{a^{}}^{1-\varepsilon,\Psi} \widetilde{x}_1(a) \right) \leq \Phi_{f_2}^{\infty} \left( I_{a^{}}^{1-\varepsilon,\Psi} \widetilde{x}_2(a) \right) \leq 
\Phi_{f_3}^{\infty} \left( I_{a^{}}^{1-\varepsilon,\Psi} \widetilde{x}_3(a) \right) .
$$

However, as $x_i = \Phi_{f_i}^{\infty} \left( I_{a^{}}^{1-\varepsilon,\Psi} \widetilde{x}_i(a) \right)$ with $i=1,2,3$, then, using {\rm Lemma \ref{k4}}, we concluded that $x_1 \leq x_2 \leq x_3$.
\end{proof}

Consider the Cauchy type problem for $\Psi-$H fractional derivative, {\rm Eq.(\ref{I1})} subject to {\rm Eq.(\ref{I1a})}, and suppose the conditions of {\rm Theorem \ref{theorem2}} are satisfied. Denoting by $x^{*}(\cdot,x_0,\mathcal{U},f)$ the solution of this problem, we have the following result.

\begin{theorem} [Data dependence theorem] \label{t5} Let $x_{0_i}, \mathcal{U}_i, f_i$ with $i=1,2$, satisfy the conditions ${\rm (C_1)}$, ${\rm (C_2)}$ and suppose that there exists $\eta_i > 0$, with $i=1,2,3$, such that
\begin{flushleft}
\begin{tabular}{cl}
{\rm (a)} & $\left\vert x_{0_{1}}(t)-x_{0_{2}}(t)\right\vert \leq \eta_{1},~ t\in\Delta$; \\ 
{\rm (b)} & $\left\vert \mathcal{U}_{1}(u)(t)-\mathcal{U}_{2}(u)(t)\right\vert \leq \eta _{2},~
t\in\Delta,~ u\in C_{\xi, \,\mu, \, \Psi }\left( \Delta,\mathbb{R}\right)$; \\ 
{\rm (c)} & $\left\vert f_{1}(t,v)-f_{2}(t,v)\right\vert \leq \eta
_{3},~ t\in\Delta,~ v\in \mathbb{R}$.%
\end{tabular}
\end{flushleft}
Then, the following inequality holds 
\begin{align} \label{ineq28}
&\left\Vert x_{1}^{\ast }(t,x_{0_{1}},\mathcal{U}_{1},f_{1})-x_{2}^{\ast }\left(t,x_{0_{2}},\mathcal{U}_{2},f_{2}\right) \right\Vert _{_{\xi, \,\mu, \, \Psi }} \nonumber\\
&\qquad\leq \frac{c}{\mathcal{E}_{\mu }(\xi (\Psi (b)-\Psi (a))^{\mu})} 
\left\lbrace \frac{(\Psi (b)-\Psi (a))^{\varepsilon-1 }}{\Gamma (\varepsilon)}\,\eta _{1}
+  \frac{(\Psi (b)-\Psi (a))^{\mu }}{\Gamma (\mu +1)}\, (\eta _{2}+\eta _{3})\right\rbrace
\end{align}
where $c$ is given by {\rm Eq.(\ref{jose})} and $x_{i}^{\ast}(t,x_{0_{i}},\mathcal{U}_{i},f_{i})$, with $i=1,2$ are the solutions of the problem, {\rm Eq.\eqref{I1}}, satisfying {\rm Eq.(\ref{I1a})}, with respect to $x_{0_{i}},\mathcal{U}_{i},f_{i}$, $L_{f}=\max \{L_{f_{1}},L_{f_{2}}\}$ and $L_{\mathcal{U}}=\max \{L_{\mathcal{U}_{1}},L_{\mathcal{U}_{2}}\}$.
\end{theorem}

\begin{proof}
For the proof, we consider the operators 
\begin{align}
\Theta _{x_{0_{i}},\mathcal{U}_{i},f_{i}}(x)(t)
&=\mathcal{R}^{\Psi}_{\varepsilon}\left( t,a\right) x_{0_{i}}+\frac{1}{\Gamma (\mu )}\int_{a}^{t}\mathcal{G}_{\Psi }^{\mu }\left(t,s\right) \mathcal{U}_{i}(x)(s)\,\mathrm{d}s +\frac{1}{\Gamma (\mu )}\int_{a}^{t}\mathcal{G}_{\Psi }^{\mu }\left( t,s\right) f_{i}(s,x(s))\,\mathrm{d}s
\end{align}
for $i=1,2$. By means of the {\rm Theorem \ref{theorem2}} these operators are $c_{i}$-Picard operators with $c_{i}$ is given by {\rm Eq.(\ref{jose})}. On the other hand, we can write
\begin{eqnarray}
&&\left| \Theta _{x_{0_{1}},\mathcal{U}_{1},f_{1}}(x)(t)-\Theta _{x_{0_{2}},\mathcal{U}_{2},f_{2}}(x)(t)\right|  \notag \\
&=&\left|\mathcal{R}^{\Psi}_{\varepsilon}\left( t,a\right) x_{0_{1}}-\mathcal{R}^{\Psi}_{\varepsilon}\left( t,a\right)x_{0_{2}}\right.   \notag \\
&&+\frac{1}{\Gamma (\mu )}\int_{a}^{t}\mathcal{G}_{\Psi }^{\mu }\left( t,s\right) \mathcal{U}_{1}(x)(s)\,\mathrm{d}s-\frac{1}{\Gamma (\mu )}\int_{a}^{t}\mathcal{G}_{\Psi }^{\mu }\left( t,s\right) \mathcal{U}_{2}(x)(s)\,\mathrm{d}s \notag \\
&&\left. +\frac{1}{\Gamma (\mu )}\int_{a}^{t}\mathcal{G}_{\Psi }^{\mu }\left(t,s\right) f_{1}(s,x(s))\,\mathrm{d}s-\frac{1}{\Gamma (\mu )}\int_{a}^{t}\mathcal{G}_{\Psi }^{\mu }\left( t,s\right) f_{2}(s,x(s))\,\mathrm{d}s\right|  \notag \\
&\leq &\mathcal{R}^{\Psi}_{\varepsilon}\left( t,a\right) \eta _{1}+\frac{1}{\Gamma (\mu )}\int_{a}^{t}\mathcal{G}_{\Psi }^{\mu }\left( t,s\right) \left| \mathcal{U}_{1}(x)\left(
s\right) -\mathcal{U}_{2}(x)\left( s\right) \right|\,\mathrm{d}s  \notag\\
&&+\frac{1}{\Gamma (\mu )}\int_{a}^{t}\mathcal{G}_{\Psi }^{\mu }\left(t,s\right) \left| f_{1}(s,x\left( s\right) )-f_{2}\left( s,x\left( s\right) \right) \right|\,\mathrm{d}s  \notag \\
&\leq &\mathcal{R}^{\Psi}_{\varepsilon}\left( t,a\right) \eta _{1}+\frac{\eta _{2}}{\Gamma 	(\mu )}\int_{a}^{t}\mathcal{G}_{\Psi }^{\mu }\left( t,s\right) \,\mathrm{d}s+\frac{\eta _{3}}{\Gamma (\mu )}\int_{a}^{t}\mathcal{G}_{\Psi }^{\mu }\left(t,s\right) \,\mathrm{d}s\notag \\
&= &\mathcal{R}^{\Psi}_{\varepsilon}\left( t,a\right) \eta _{1}
+(\eta _{2}+\eta _{3}) \,\mathbf{I}_{a+}^{\mu ,\Psi }(1) \notag \\
&\leq & \frac{(\Psi (t)-\Psi (a))^{\varepsilon-1 }}{\Gamma (\varepsilon)}\,\eta _{1}
+  \frac{(\Psi (t)-\Psi (a))^{\mu }}{\Gamma (\mu +1)}\, (\eta _{2}+\eta _{3}) \notag.
\end{eqnarray}
Therefore,
\begin{eqnarray*}
&&\left\Vert \Theta _{x_{0_{1}},\mathcal{U}_{1},f_{1}}(x)-\Theta _{x_{0_{2}},\mathcal{U}_{2},f_{2}}(x)\right\Vert _{\xi, \,\mu, \, \Psi }\\ &=&\sup_{t\in \Delta}\frac{\left| \Theta _{x_{0_{1}},\mathcal{U}_{1},f_{1}}(x)(t)-\Theta _{x_{0_{2}},\mathcal{U}_{2},f_{2}}(x)(t)\right| }{\mathcal{E}_{\mu }(\xi (\Psi (t)-\Psi (a))^{\mu})}\\
&\leq& \frac{1}{\mathcal{E}_{\mu }(\xi (\Psi (b)-\Psi (a))^{\mu})} 
\left\lbrace \frac{(\Psi (b)-\Psi (a))^{\varepsilon-1 }}{\Gamma (\varepsilon)}\,\eta _{1}
+  \frac{(\Psi (b)-\Psi (a))^{\mu }}{\Gamma (\mu +1)}\, (\eta _{2}+\eta _{3})\right\rbrace \notag
\end{eqnarray*}
The proof of the inequality \eqref{ineq28} follows by applying the {\rm Lemma \ref{k6}}.
\end{proof}

The Theorem \ref{jr} that will be investigated below, is a direct application of Theorem 2 (See \cite{pri1}) which deals with the functional Pompeiu-Hausdorff operator.
									
\begin{theorem}\label{jr} Assume that $f_{1},f_{2}:C_{\xi, \,\mu, \, \Psi }\left(\Delta\, , \mathbb{R}\right) \rightarrow \mathbb{R}$ satisfy the conditions $(C_{1})$ and $(C_{2})$, given in {\rm Eq.(\ref{C1})} and {\rm Eq.(\ref{C2})}. Let $S_{\Phi _{f_{1}}},S_{\Phi_{f_{2}}}$ be the solution set of system {\rm Eq.(\ref{I1})}, satisfying {\rm Eq.(\ref{I1a})} corresponding to $f_{1}$ and $f_{2}$. Suppose that, there exist $\eta _{i}>0$, with $i=1,2$, such that 
\begin{equation}
\left\vert \mathcal{U}_{1}(u)(t)-\mathcal{U}_{2}(t)(u)\right\vert \leq \eta _{1}\quad \mathrm{and}\quad \left\vert f_{1}(t,v)-f_{2}(t,v)\right\vert \leq \eta _{2}  \label{3.4}
\end{equation}
for $t\in\Delta$, $u\in C_{\xi, \,\mu, \, \Psi }\left( \Delta,\mathbb{R}\right)$ and $v\in \mathbb{R}$. Then, 
\begin{equation}
H_{\left\Vert \cdot \right\Vert _{\xi, \,\mu, \, \Psi }}\left( S_{\Phi_{f_{1}}},S_{\Phi _{f_{2}}}\right) \leq \frac{c\,(\eta _{1}+\eta _{2})\, (\Psi (b)-\Psi (a))^{\mu }}{\Gamma (\mu +1)\,\mathcal{E}_{\mu }(\xi (\Psi (b)-\Psi (a))^{\mu})} ,
\end{equation}
where  $c$ is given by {\rm Eq.(\ref{jose})},  $L_{f}=\max \{L_{f_{1}},L_{f_{2}}\}$, $L_{\mathcal{U}}=\max\{L_{\mathcal{U}_{1}},L_{\mathcal{U}_{2}}\}$ and $H_{\left\Vert  \cdot  \right\Vert_{\xi, \,\mu, \, \Psi }}$ denotes the Pompeiu-Hausdorff functional with respect to $\left\Vert  \cdot  \right\Vert _{\xi, \,\mu, \, \Psi }$ on $C_{\xi, \,\mu, \, \Psi }\left(\Delta\, , \mathbb{R}\right)$.
\end{theorem}

\begin{proof}
Let $f_{1}:C_{\xi, \,\mu, \, \Psi }(\Delta\, , \mathbb{R})\rightarrow \mathbb{R}$ and $f_{2}:C_{\xi, \,\mu, \, \Psi }(\Delta\, , \mathbb{R})\rightarrow \mathbb{R}$ be two function and $S_{\Phi _{f_{1}}},S_{\Phi _{f_{2}}}$ a solution set with $\Phi_{f_{1}}$ and $\Phi _{f_{2}}$ given by 
\begin{align*}
\Phi _{f_{1}}x(t)&=\mathcal{R}^{\Psi}_{\varepsilon}\left( t,a\right) \mathbf{I}_{a^{+}}^{1-\xi ,\Psi}x(a)+\frac{1}{\Gamma (\mu )}\int_{a}^{t}\mathcal{G}_{\Psi }^{\mu }\left(t,s\right) \mathcal{U}_{1}(x)(s)\,\mathrm{d}s+\frac{1}{\Gamma (\mu )}\int_{a}^{t}\mathcal{G}_{\Psi }^{\mu }\left( t,s\right) f_{1}(s,x(s))\,\mathrm{d}s
\end{align*}
and 
\begin{align*}
\Phi _{f_{2}}x(t)&=\mathcal{R}^{\Psi}_{\varepsilon}\left( t,a\right) \mathbf{I}_{a^{+}}^{1-\xi ,\Psi}x(a)+\frac{1}{\Gamma (\mu )}\int_{a}^{t}\mathcal{G}_{\Psi }^{\mu }\left(t,s\right) \mathcal{U}_{2}\left( x\right) (s)\,\mathrm{d}s+\frac{1}{\Gamma (\mu )}\int_{a}^{t}\mathcal{G}_{\Psi }^{\mu }\left( t,s\right) f_{2}(s,x(s))\,\mathrm{d}s.
\end{align*}
Using the hypotheses given by {\rm Eq.(\ref{3.4})}, we have 

\begin{align}
H_{\left\Vert \cdot \right\Vert _{\xi, \,\mu, \, \Psi }}\left( S_{\Phi_{f_{1}}},S_{\Phi _{f_{2}}}\right) &=\max \left\lbrace  \sup_{t\in S_{\Phi_{f_{1}}}}\inf_{t\in S_{\Phi _{f_{2}}}}\left\Vert \Phi _{f_{1}}(x)-\Phi_{f_{2}}(x)\right\Vert _{_{\xi, \,\mu, \, \Psi }} \,\,, \right.\left. \,\,\sup_{t\in S_{\Phi_{f_{2}}}}\inf_{t\in S_{\Phi _{f_{1}}}}\left\Vert \Phi _{f_{1}}(x)-\Phi_{f_{2}}(x)\right\Vert _{_{\xi, \,\mu, \, \Psi }}\right\rbrace  .  \label{P-1}
\end{align}
Noting that
\begin{eqnarray}\label{P-2}
\left\vert \Phi _{f_{1}}(x)(t)-\Phi _{f_{2}}(x)(t)\right\vert &=&\left\vert \frac{1}{\Gamma (\mu )}\int_{a}^{t}\mathcal{G}_{\Psi }^{\mu}\left( t,s\right) \mathcal{U}_{1}(x)(s)\,\mathrm{d}s-\frac{1}{\Gamma (\mu )}\int_{a}^{t}\mathcal{G}_{\Psi }^{\mu }\left( t,s\right) \mathcal{U}_{2}(x)(s)\,\mathrm{d}s\right.   \notag \\
&&\left. +\frac{1}{\Gamma (\mu )}\int_{a}^{t}\mathcal{G}_{\Psi }^{\mu }\left(t,s\right) f_{1}(s,x(s))\,\mathrm{d}s-\frac{1}{\Gamma (\mu )}\int_{a}^{t}\mathcal{G}_{\Psi }^{\mu }\left( t,s\right) f_{2}(s,x(s))\,\mathrm{d}s\right\vert \notag \\
&\leq &\frac{1}{\Gamma (\mu )}\int_{a}^{t}\mathcal{G}_{\Psi }^{\mu }\left(t,s\right) \left\vert \mathcal{U}_{1}(x)(s)-\mathcal{U}_{2}(x)(s)\right\vert \,\mathrm{d}s  \notag \\
&&+\frac{1}{\Gamma (\mu )}\int_{a}^{t}\mathcal{G}_{\Psi }^{\mu }\left(t,s\right) \left\vert f_{1}(s,x(s))-f_{2}(s,x(s))\right\vert \,\mathrm{d}s  \notag \\
&\leq &\frac{\eta _{1}}{\Gamma (\mu )}\int_{a}^{t}\mathcal{G}_{\Psi }^{\mu}\left( t,s\right) \,\mathrm{d}s+\frac{\eta _{2}}{\Gamma (\mu )}
\int_{a}^{t}\mathcal{G}_{\Psi }^{\mu }\left( t,s\right) \mathrm{d}s  \notag \\
&=&(\eta _{1}+\eta _{2})\frac{(\Psi (t)-\Psi (a))^{\mu }}{\Gamma (\mu+1)}.
\end{eqnarray}
Therefore, for any $x,y\in \Omega $, 
\begin{eqnarray*}
\left\Vert \Phi _{f_{1}}(x)-\Phi _{f_{2}}(x)\right\Vert _{\xi, \,\mu, \, \Psi } &=&\sup_{t\in \Delta}\frac{\left| \Phi _{f_{1}}(x)(t)-\Phi _{f_{2}}(x)(t)\right| }{\mathcal{E}_{\mu }(\xi (\Psi (t)-\Psi (a))^{\mu})}\\
&\leq& \frac{(\eta _{1}+\eta _{2})}{\mathcal{E}_{\mu }(\xi (\Psi (b)-\Psi (a))^{\mu})} \times \frac{(\Psi (b)-\Psi (a))^{\mu }}{\Gamma (\mu +1)}\notag
\end{eqnarray*}
Substituting {\rm Eq.(\ref{P-2})} in {\rm Eq.(\ref{P-1})} and following the same steps as in {\rm Theorem \ref{theorem2}}, we conclude that 
\begin{equation}
H_{\left\Vert  \cdot  \right\Vert _{_{\xi, \,\mu, \, \Psi }}}\left( S_{\Phi_{f_{1}}},S_{\Phi _{f_{2}}}\right) \leq  \frac{c\,(\eta _{1}+\eta _{2})\, (\Psi (b)-\Psi (a))^{\mu }}{\Gamma (\mu +1)\,\mathcal{E}_{\mu }(\xi (\Psi (b)-\Psi (a))^{\mu})} 
\end{equation}
where $c$ is given by {\rm Eq.(\ref{jose})},   $L_{\mathcal{U}}=\max \{L_{\mathcal{U}_{1}},L_{\mathcal{U}_{2}}\}$ and $L_{f}=\max\{L_{f_{1}},L_{f_{2}}\}$. 
\end{proof}

For an example we will consider the following fractional functional differential equation 
\begin{equation}
{}^{\mathbf{H}}\mathfrak{D}_{a^{+}}^{\mu ,\nu ,\Psi }x(t)=\frac{1}{\Gamma (\mu )}\int_{0}^{t}\mathcal{G}_{\Psi }^{\mu }\left( t,s\right) \mathcal{A}(t,s,x(s),x(\lambda s)\,)
\mathrm{d}s+f(t,x(t)),  \label{4.1}
\end{equation}%
with $t\in \Delta_{1}$ $(\Delta_{1}:=[0,1])$ and ${}^{\mathbf{H}}\mathfrak{D}_{a^{+}}^{\mu ,\nu ,\Psi}\left( \cdot \right) $ is the $\Psi-$H fractional derivative.

Some conditions will be necessary for the investigation of {\rm Eq.(\ref{4.1})}. Then, consider the following conditions:

{\rm (A)} $\lambda \in (0,1)$, $\mathcal{A}:C_{\xi, \,\mu, \, \Psi }(\Delta_{1}\times \Delta_{1}\times \mathbb{R}^{2}, \, \mathbb{R})\rightarrow \mathbb{R}$, $f:C_{\xi, \,\mu, \, \Psi }(\Delta_{1},\,  \mathbb{R}%
)\rightarrow \mathbb{R}$;

{\rm (B)} $\exists \, L_{f}>0$ such that $|f(t,u_{1})-f(t,u_{2})|\leq L_{f}|u_{1}-u_{2}|$, with $t\in \Delta_{1}$ and $u_{1},u_{2}\in \mathbb{R}$%
;

{\rm (C)} $\exists \, L_{\mathcal{A}}>0$ such that 
\begin{equation}
\left\vert \mathcal{A}(t,s,u_{1},u_{2})-\mathcal{A}(t,s,v_{1},v_{2})\right\vert \leq L_{\mathcal{A}}\left(\left\vert u_{1}-v_{1}\right\vert +\left\vert u_{2}-v_{2}\right\vert \right) 
\end{equation}
with $t,s\in \Delta_{1}$ and $u_{1},v_{1},u_{2},v_{2}\in \mathbb{R}$;

{\rm (D)} $\exists \, \xi >0$ such that, 

\begin{equation}
\frac{L_f}{\xi}+ \frac{2 L_{\mathcal{A}}}{\xi} \,\left( \frac{1}{\xi}-\frac{\left( \Psi \left( 1\right) -\Psi \left( 0\right) \right) ^{\mu }}{\Gamma \left( \mu +1\right) }\right) < 1.
\end{equation}
Note that $x\in C_{\xi, \,\mu, \, \Psi }^{1}(\Delta_{1},\mathbb{R})$ is a solution of {\rm Eq.(\ref{4.1})} iff  $x\in C_{\xi, \,\mu, \, \Psi }(\Delta_{1},\mathbb{R})$ is a solution of the following fractional integral equation
\begin{align}\label{4.2}
x(t)&=\mathcal{R}^{\Psi}_{\varepsilon}\left( t,0\right) \mathbf{I}_{0^{+}}^{\mu ,\Psi}x(0)+\mathbf{I}_{0^{+}}^{\mu ,\Psi }f(t,x(t)) +\mathbf{I}_{0^{+}}^{\mu ,\Psi }\left(\frac{1}{\Gamma (\mu )}\int_{0}^{t}\mathcal{G}_{\Psi }^{\mu }\left( t,s\right)\mathcal{A}(t,s,x(s),x(\lambda s))\mathrm{d}s\right),~t\in \Delta_{1} 
\end{align}%
Consider the operator $\Phi _{f}:C_{\xi, \,\mu, \, \Psi }([0,1],\mathbb{R})\rightarrow C_{\xi, \,\mu, \, \Psi }([0,1],\mathbb{R})$ defined by 
\begin{align}
\Phi _{f}(x)(t)&=\mathcal{R}^{\Psi}_{\varepsilon}\left( t,0\right) \mathbf{I}_{0^{+}}^{\mu ,\Psi}x(0)+\mathbf{I}_{0^{+}}^{\mu ,\Psi }f(t,x(t))+\mathbf{I}_{0^{+}}^{\mu ,\Psi }\left( \frac{1}{\Gamma (\mu )}\int_{0}^{t}\mathcal{G}_{\Psi }^{\mu }\left( t,s\right)\mathcal{A}(t,s,x(s),x(\lambda s))\mathrm{d}s\right),~t\in \Delta_{1} 
\end{align}%

Consider on $\Omega =C_{\xi, \,\mu, \, \Psi }(\Delta_{1},\mathbb{R})$ the $({\xi, \,\mu, \, \Psi }) $-norm defined as  and for $\delta \in \mathbb{R}$ the set $\Omega _{\delta }:=\{x\in C_{\xi, \,\mu, \, \Psi }(\Delta_{1},\mathbb{R}): I_{0^{+}}^{1-\varepsilon ,\Psi }x(0)=\delta \}$. Some observations of the sets $\Omega$ and $\Omega_\delta$ can be found in the development of the paper.

From the conditions of {\rm Theorem \ref{theorem2}} we have the operator $\Phi_f$ is weakly Picard operator in $C_{\xi, \,\mu, \, \Psi }(\Delta_{1},\mathbb{R})$. Additionally, one can apply  {\rm Theorem \ref{t3}}, Theorem \ref{t4} and Theorem \ref{t5} for the investigation  of $\check{C}$aplygin  inequalities, monotonicity and data dependency of the solution of {\rm Eq.(\ref{4.1})}.

An important property of the $\Psi-$H fractional derivative is that it include the tremendous class of fractional derivatives and with them their properties which can be acquired as specific cases of function $\Psi$ and parameters $\mu$ and $\nu$. 

Choosing $\Psi(t)=t$ and taking the limit $\nu \to 1$ on both sides of {\rm Eq.(\ref{4.1})}, we obtain the following fractional differential equation 
\begin{equation}  \label{Eq.*}
D_{0^{+}}^{\mu} x(t) = \frac{1}{\Gamma(\mu)} \int_0^t (t-s)^{\mu-1} \mathcal{A}(t,s,x(s),x(\lambda s))\, \mathrm{d}s + f(t,x(t))
\end{equation}
where $D_{0^{+}}^{\mu} (\cdot)$ is the Caputo fractional derivative of order $\mu$, with $0 < \mu \leq 1$.

Assuming the conditions {\rm (A)}, {\rm (B)}, {\rm (C)} and {\rm (D)}. By means of these assumptions {\rm Eq.(\ref{4.2})} is equivalent to an integral equation 
\begin{align}
x(t)&=\frac{t^{\varepsilon -1}}{\Gamma (\varepsilon )}\mathbf{I}_{0^{+}}^{1-\varepsilon}x(0)+\mathbf{I}_{0^{+}}^{1-\varepsilon }f(t,x(t)) +\mathbf{I}_{0^{+}}^{\mu }\left( \frac{1}{\Gamma (\mu )}\int_{0}^{t}(t-s)^{\mu -1}\mathcal{A}(t,s,x(s),x(\lambda s))\,\mathrm{d}s\right), ~t\in \Delta_{1}.
\end{align}%

Let us consider the following operator $\widetilde{\Phi }_{f}:C_{\xi, \,\mu, \, \Psi }(\Delta_{1},\mathbb{R})\rightarrow C_{\xi, \,\mu, \, \Psi }(\Delta_{1},\mathbb{R})$ defined by 
\begin{align}
\widetilde{\Phi }_{f}(x)(t)&=\frac{t^{\varepsilon -1}}{\Gamma (\varepsilon )}\mathbf{I}_{0^{+}}^{1-\varepsilon}x(0)+\mathbf{I}_{0^{+}}^{1-\varepsilon }f(t,x(t)) +\mathbf{I}_{0^{+}}^{\mu }\left( \frac{1}{\Gamma (\mu )}\int_{0}^{t}(t-s)^{\mu -1}\mathcal{A}(t,s,x(s),x(\lambda s))\,\mathrm{d}s\right), ~t\in \Delta_{1}.
\end{align}%
On the other hand, consider on $X=C_{\xi, \,\mu}(\Delta_{1},\mathbb{R})$ the $ ({\xi, \,\mu}) $--norm $\left\Vert  \cdot  \right\Vert _{\xi, \,\mu }$ defined as
\begin{equation}
\left\Vert x\right\Vert _{\xi, \,\mu,}=\sup_{t\in \Delta_1}\frac{\left | x\left(t\right) \right | }{\mathcal{E}_{\mu } \left(  \xi t
^{\mu }\right)  }.  
\end{equation}
 Note that if  we choose $\Psi \left( t\right) =t$ and ~ $({\xi, \,\mu,})-$ norm, we have a particular case of norm. However, it still remains a $({\xi, \,\mu,})-$ norm. Other particular cases, as noted above, simply choose $\Psi \left(\cdot \right) $, $\mu $ and the limits $\nu \rightarrow 1$ and  $\nu\rightarrow 0.$

\section{Concluding remarks}

In the present paper, we presented a new class of the  fractional functional differential equation with the abstract operator of Volterra. From this new class fractional functional differential equation, we investigated the existence, uniqueness and data dependence, and the results presented were in the context of the Picard operators. Provided an example to elucidate the results obtained. The new results obtained here by means of the $\Psi-$H fractional derivative, actually contribute to the study of fractional differential equations, as providing a new class of solutions that may be important in obtaining new results.

\section{Acknowledgment}

\textbf{ I have been financially supported  by  PNPD-CAPES scholarship of the Pos-Graduate Program in Mathematics Applied IMECC-Unicamp.}  

\bibliography{ref}
\bibliographystyle{plain}

\end{document}